\documentclass[12pt]{article}

\usepackage[margin=1in]{geometry}  
\usepackage{ amssymb }
\usepackage{graphicx}
\usepackage{tikz}      
\usepackage{amsmath}               
\usepackage{amsthm}                
\usepackage[ansinew]{inputenc}
\usepackage{array,color,amscd}
\usepackage{amsxtra}
\usepackage{amstext}
\usepackage{latexsym}
\usepackage{dsfont}
\usepackage[all]{xy}
\usepackage{amsmath,amscd,amsfonts,amssymb}
\usepackage{cite}
\usetikzlibrary{matrix,arrows,decorations.pathmorphing}

\newtheorem{thm}{Theorem}[section]
\newtheorem{cor}[thm]{Corollary}
\newtheorem{lem}[thm]{Lemma}
\newtheorem{prop}[thm]{Proposition}
\theoremstyle{definition}

\theoremstyle{remark}

\theoremstyle{example}

\numberwithin{equation}{section}

\begin{document}

\title{The homotopy types of $\mathrm{Sp(n)}$-gauge groups over $S^{4m}$ }
\author{Sajjad Mohammadi}

\date{}
\maketitle
\begin{abstract}
Let $m$ and $n$ be two positive integers such that $m < n$, and let $Q_{n-m+1}$ be the symplectic quasi-projective
space
of rank $n-m+1$. Denote by $P_{n,k}$ the principal $\mathrm{Sp(n)}$-bundle over $S^{4m}$ and
$\mathcal{G}_{k,m}(\mathrm{Sp(n)})$ be the gauge group of $P_{n,k}$ classified by $k\varepsilon'$, where
$\varepsilon'$ is a generator of $\pi_{4m}(B(\mathrm{Sp(n)}))\cong\mathbb{Z}$. In this article, first, we will study
the order of the Samelson product $S^{4m-1}\wedge Q_{n-m+1}\rightarrow \mathrm{Sp(n)}$. Then, by the unstable
$K$-theory method, we will partially classify the homotopy types of $\mathcal{G}_{k,m}(\mathrm{Sp(n)})$ by giving a
lower bound for the number of homotopy types of $\mathcal{G}_{k,m}(\mathrm{Sp(n)})$.
\end{abstract}
$2010$ Mathematics Subject Classification: Primary $55Q15$; Secondary $55P10; 55P15$.\\\\
\section{Introduction }\label{a0}
For two maps $f_1\colon  A_1 \rightarrow X$ and $f_2\colon A_2\rightarrow X$ into a homotopy associative $H$-space $X$
with inverse, the Samelson product of $f_1$ and $f_2$ is defined as the following composition
$$ A_1\wedge A_2 \overset{f_1\wedge f_2}\longrightarrow  X\wedge X \overset{c}\longrightarrow X $$
and is denoted by $\langle f_1, f_2 \rangle$, where the map $c$ is the reduced commutator map. Let $G$ be a
simply-connected, simple compact Lie group. The calculation of Samelson products plays an important role in
classifying the homotopy types of gauge groups of principal $G$-bundles. Also, they are fundamental in studying the
homotopy nilpotency and the homotopy commutativity of Lie groups. Focusing on the groups $\mathrm{Sp(n)}$, in the
1960s Bott \cite{[B]} gave some information on the orders of Samelson products of maps representing the integral
homotopy groups. In the 1980s, McGibbon \cite{[M]} classified at what primes and for which $n$ the group
$\mathrm{Sp(n)}$ is homotopy commutative (and so have all of its Samelson products been trivial). In the late 2000s,
Hamanaka, Kaji and Kono
\cite{[HKK]} considered the map $\epsilon_n\colon Q_n \rightarrow \mathrm{Sp(n)}$, where $Q_n$ is James
quasi-projective space, and has the property that $H_*(\mathrm{Sp(n)}) \cong \Lambda({\tilde{H}}_*(\mathrm{Sp(n)}))$.
They
determined the order of $\langle\epsilon_2, \epsilon_2\rangle$ at the prime three. In following some related work by
them, Nagao \cite{[N]} showed that the Samelson product of $\epsilon_2$ with itself has order $280$, and determined
the odd primary components of the orders of the Samelson products of $\epsilon_n$ with itself for $n\in \{3, 4, 5\}$.
Moreover, Samelson products have been studied extensively for the classical groups. In \cite{[H]}, Hamanaka studies
the Samelson products of $\mathrm{U(n)}$ localized at a prime $p$. The Samelson products of the inclusions of the
spheres into $p$-regular exceptional Lie groups were completely determined in \cite{[HKO]}. In \cite{[O]}, Oshima
determines the Samelson product
$$\pi_n(G_2)\times \pi_{11}(G_2)\rightarrow \pi_{n+11}(G_2), $$
for $n=3, 11$. Also in \cite{[KKKT]}, Kamiyama, Kishimoto, Kono, and Tsukuda have studied the Samelson products of
$\mathrm{SO(3)}$, and in \cite{[M2]}, the author has studied the order of Samelson products in $\mathrm{Sp(n)}$.\par
Let $P\rightarrow B$ be a principal $G$-bundle over connected finite complex $B$. The gauge group $\mathcal{G}(P)$ of
$P$ is the group of $G$-equivariant automorphisms of $P$ which fix $B$. A detailed introduction to the topology of
gauge groups of bundles can be found in \cite{[H2]}, \cite{[PS]}. The topology of gauge groups plays an important role
in the physics and identification of four-dimensional manifolds. Crabb and Sutherland \cite{[CS]} have shown that if
$B$ and $G$ are as above, then the number of homotopy types amongst all the gauge groups of principal $G$-bundles over
$B$ is finite. This is even though the number of isomorphism classes of principal $G$-bundles over $B$ is often
infinite. For integers $a$ and $b$ let $(a, b)$ be the greatest common divisor of $|a|$ and $|b|$. There is a
one-to-one correspondence between equivalence classes of principal $G$-bundles over $B=S^4$ and homotopy classes of
maps $[S^4, BG]\cong  \mathbb{Z}$. The first classification was done by Kono \cite{[K_1]} for $G = \mathrm{SU(2)}$.
Kono showed that there is a homotopy equivalence $\mathcal{G}_k \simeq \mathcal{G}_{k'}$ if and only if
$(12,k)=(12,k')$. It thus follows that there are precisely six homotopy types of $\mathrm{SU(2)}$-gauge groups over
$S^4$. In recent years, the determination of the precise number of homotopy types of gauge groups for some Lie groups
on manifolds with different dimensions interested topologists. For principal bundles over $S^4$ with different
structure groups, they obtained some useful results that we in here mention some of their results.
$\mathrm{U(n)}$-gauge groups (see \cite{[C]}); $\mathrm{PU(3)}$- and $\mathrm{PSp(2)}$-gauge groups (see
\cite{[HKKT]});  The gauge groups of exceptional Lie groups (see \cite{[HKST]}); $\mathrm{SO(4)}$-gauge groups (see
\cite{[KMT]}); $\mathrm{SU(n)}$-gauge groups (see \cite{[T_3]}). Also, there are also several classification results
for gauge groups of principal bundles with base spaces other than $S^4$, such as the gauge groups of quaternionic line
bundles over spheres (see \cite{[CS1]}); The gauge groups over high-dimensional manifolds (see \cite{[H1]});
$\mathrm{SU(n)}$-gauge groups over $S^{2m}$ (see \cite{[M0]}); $\mathrm{PU(5)}$- and $\mathrm{PU(3)}$-gauge groups
over $S^4$ and $S^6$, respectively (see \cite{[R]}).\par
Let $m$ and $n$ be two positive integers such that $m < n$. Let $Q_{n-m+1}$ be the symplectic quasi-projective space
of rank $n-m+1$, we denote the inclusion $Q_{n-m+1}\rightarrow \mathrm{Sp(n)}$ by $\epsilon_{m,n}$. Also, let
$\varepsilon_{m,n}\colon S^{4m-1} \rightarrow \mathrm{Sp(n)}$ represent the generator of
$\pi_{4m-1}(\mathrm{Sp(n)})$. In this article, we will study the order of the Samelson product $\langle
\varepsilon_{m,n},\epsilon_{m,n} \rangle\colon S^{4m-1}\wedge Q_{n-m+1}\rightarrow \mathrm{Sp(n)}$.  We will prove the
following theorem.
\begin{thm}\label{a1}
\rm{ The order of the Samelson product $\langle \varepsilon_{m,n},\epsilon_{m,n} \rangle\colon S^{4m-1}\wedge
Q_{n-m+1}\rightarrow \mathrm{Sp(n)} $ is equal to
\begin{displaymath}
 \left\{ \begin{array}{ll}
\frac{(2n+1)!}{2(2n-2m+1)!(2m-1)!} & \textrm{if $m$ is even },\\\\
\frac{2(2n+1)!}{(2n-2m+1)!(2m-1)!} & \textrm{if $m$ is odd }.\\
\end{array} \right.
\end{displaymath}
}
\end{thm}
The theorem recovers the known case in \cite{[KK]} when $m=1$.\par
In \cite{[KK]} Kishimoto and Kono classify the homotopy types of $\mathrm{Sp(n)}$-gauge groups over $S^4$, by showing
that if there is a homotopy equivalence $\mathcal{G}_{k,1}(\mathrm{Sp(n)})\simeq \mathcal{G}_{k',1}(\mathrm{Sp(n)})$
then $(4n(2n+1),k)=(4n(2n+1),k')$. In this article, in general case, we study the classification of the homotopy types
of the gauge groups of principal $\mathrm{Sp(n)}$-bundles over $S^{4m}$ and we give only lower bounds for the number
of homotopy types and does not prove the converse, which is very hard and realistically out of reach. In only a few
cases, we have an upper bound for the number of homotopy types of $\mathcal{G}_{k,m}(\mathrm{Sp(n)})$.
\begin{itemize}
  \item $n=2, m=1\colon $  if $(40,k)=(40,k')$, then $\mathcal{G}_{k,1}(\mathrm{Sp(2)})\simeq
      \mathcal{G}_{k',1}(\mathrm{Sp(2)})$ when localized rationally or at any prime \cite{[T_2]};
  \item $n=3, m=1\colon$ if $(336,k)=(336,k')$, then $\mathcal{G}_{k,1}(\mathrm{Sp(3)})\simeq
      \mathcal{G}_{k',1}(\mathrm{Sp(3)})$ after rationalisation or localisation at any prime \cite{[C1]};
  \item $n=2, m=2\colon$  if $(140,k)=(140,k')$, then $\mathcal{G}_{k,2}(\mathrm{Sp(2)})\simeq
      \mathcal{G}_{k',2}(\mathrm{Sp(2)})$ \cite{[HKK]};
  \item $n=3, m=2\colon$ if $(2835, k)=(2835, k')$ then
      $\mathcal{G}_{k,2}(\mathrm{Sp(3)})\simeq\mathcal{G}_{k',2}(\mathrm{Sp(3)})$ \cite{[M1]}.
\end{itemize}
In other cases, especially for $m\geq 3$, calculating the upper bound for the number of homotopy types would be very
difficult because the structure of the related groups becomes more complicated, and the calculation of the order of
associated maps is tough, and we also have little information about them. Put $t=\frac{1}{2(2m-1)!}
(2n+1)2n(2n-1)\cdots (2n-2m+2)$. Let
$P_{n,k}$ be the principal $\mathrm{Sp(n)}$-bundles over $S^{4m}$ and $\mathcal{G}_{k,m}(\mathrm{Sp(n)})$ be the gauge
group of $P_{n,k}$ classified by $k\varepsilon'$, where $\varepsilon'$ is a generator of
$\pi_{4m}(B(\mathrm{Sp(n)}))\cong\mathbb{Z}$.
\begin{thm}\label{a2}
\rm{If $\mathcal{G}_{k,m}(\mathrm{Sp(n)})$ is homotopy equivalent to $\mathcal{G}_{k',m}(\mathrm{Sp(n)})$, then the
following hold:\\
\textbf {Case one:} if $m$ is an even integer, then $(k,t)=(k',t)$,\\
\textbf {Case two:} if $m$ is an odd integer, then
\begin{displaymath}
\left\{ \begin{array}{ll}
(k,4t)=(k',4t) & \textrm{if $n$ is even },\\\\
(k,t)=(k',t) & \textrm{if $n$ is odd }.\\
\end{array} \right.
\end{displaymath}
}
\end{thm}
In case $m=1$, by \cite{[S]}, we know that there is an result as if $\mathcal{G}_{k,1}(\mathrm{Sp(n)})$ is homotopy
equivalent to $\mathcal{G}_{k',1}(\mathrm{Sp(n)})$, then $(k, 4n(2n+1))=(k', 4n(2n+1))$ for $n$ odd. On the other
hand, Theorem \ref{a2} shows that the same result holds for $n$ even. Therefore, Theorem \ref{a2} recovers Theorem
$1.1$ in \cite{[KK]}.
\section{Preliminaries}\label{b0}
Let $\mathrm{Map}_k(S^{4m}, \mathrm{BSp(n)})$ be the component of the space of continuous unbased maps from $S^{4m}$
to $\mathrm{BSp(n)}$ which contains the map inducing $P$, similarly, let $\mathrm{Map}_k^\ast(S^{4m},
\mathrm{BSp(n)})$ be the space of pointed continuous maps from $S^{4m}$ and  $\mathrm{BSp(n)}$ which contains the map
inducing $P$. We know that there is a fibration
$$\mathrm{Map}_k^*(S^{4m}, \mathrm{BSp(n)})\rightarrow \mathrm{Map}_k(S^{4m}, \mathrm{BSp(n)})
\overset{\mathrm{ev}}\longrightarrow  \mathrm{BSp(n)},$$
where the map $\mathrm{ev}$ is evaluation map at the basepoint of $S^{4m}$. Let $B\mathcal{G}_{k,m}(\mathrm{Sp(n)})$
be the classifying space of  $\mathcal{G}_{k,m}(\mathrm{Sp(n)})$. By Atiyah-Bott \cite{[AB]}, there is a homotopy
equivalence
$$B\mathcal{G}_{k,m}(\mathrm{Sp(n)})\simeq \mathrm{Map}_k(S^{4m}, B\mathrm{Sp(n)}).$$
The evaluation fibration, therefore, determines a homotopy fibration sequence
\begin{eqnarray}
\mathrm{Sp(n)}\overset{\alpha_k}\longrightarrow \mathrm{Map}_k^*(S^{4m}, \mathrm{BSp(n)})\rightarrow
B\mathcal{G}_{k,m}(\mathrm{Sp(n)}) \overset{\mathrm{ev}}\longrightarrow \mathrm{BSp(n)}.
\end{eqnarray}
It is well known that there is a homotopy equivalence
$$\mathrm{Map}_k^*(S^{4m}, \mathrm{BSp(n)})\simeq \mathrm{Map}_0^*(S^{4m}, \mathrm{BSp(n)}),$$
we writing $\Omega_{0}^{4m-1}\mathrm{Sp(n)}$ for $\mathrm{Map}_0^*(S^{4m}, \mathrm{BSp(n)})$. For an $H$-space
$X$, let $k\colon X\rightarrow X$ be the $k$-th power map. By \cite{[L]} we have the following lemma.
\begin{lem}\label{b1}
\rm{The adjoint of the connecting map $\mathrm{Sp(n)}\overset{\alpha_k}\longrightarrow
\Omega_{0}^{4m-1}\mathrm{Sp(n)}$ is homotopic to the Samelson product $S^{4m-1} \wedge \mathrm{Sp(n)} \overset{\langle
k\varepsilon_{m,n},1\rangle}\longrightarrow \mathrm{Sp(n)}$, where $1: \mathrm{Sp(n)} \rightarrow \mathrm{Sp(n)}$ is
the identity map on $\mathrm{Sp(n)}$.} $\quad \Box$
\end{lem}
The linearity of the Samelson product implies that $\langle k\varepsilon_{m,n},1\rangle\simeq k\langle
\varepsilon_{m,n},1\rangle$. Taking adjoints, therefore, implies the following.
\begin{cor}\label{b2}
\rm{The connecting map $\alpha_k$ satisfies $\alpha_k\simeq k\circ \alpha_1$.} $\quad \Box$
\end{cor}
The organization of this article is as follows. In Section \ref{c0}, in cases where $m$ is an even integer and $m$ is
an odd integer, we will study the group $[\Sigma^{4m-1}Q_{n-m+1}, \mathrm{Sp(n)}]$. In Section \ref{d0}, regarding the
Samelson product $\langle \varepsilon_{m,n},\epsilon_{m,n} \rangle$  as an element of $[\Sigma^{4m-1}Q_{n-m+1},
\mathrm{Sp(n)}]$, we will study the order of the Samelson product $\langle \varepsilon_{m,n},\epsilon_{m,n} \rangle$.
Then, in Section \ref{e0}, we will study the group  $[\Sigma^{4n-4(m+1)}Q_2, B\mathcal{G}_{k,m}(\mathrm{Sp(n)})]$ and
will prove Theorem \ref{a2}.
\section{The group $[\Sigma^{4m-1}Q_{n-m+1}, \mathrm{Sp(n)}]$}\label{c0}
Our main goal in this section is to study the Samelson product $\langle \varepsilon_{m,n},\epsilon_{m,n} \rangle$,
where the map $\epsilon_{m,n}\colon Q_{n-m+1}\rightarrow \mathrm{Sp(n)}$ is the inclusion map. We denote
$\mathrm{Sp(\infty)}/\mathrm{Sp(n)}$ by $X_n$ and $[X, \mathrm{Sp(n)}]$ by $\mathrm{Sp_n(X)}$. Let $Q_n$ be the
symplectic quasi projective space of rank $n$ defined in \cite{[J]}. This space has the following cellular structure
$$Q_n=S^3\cup e^7\cup e^{11}\cup \cdots \cup e^{4n-1}.$$
Put $X=S^{4m-1}\wedge Q_{n-m+1}$. Note that $X$ has the following cellular structure
$$X\simeq S^{4m+2}\cup e^{4m+6}\cup e^{4m+10}\cup \cdots \cup e^{4n+2}.$$
Recall that the cohomologies of $\mathrm{Sp(n)}$, $\mathrm{Sp(\infty)}$ and $\mathrm{BSp(\infty)}$ as an algebra are
given by
\begin{align*}
&H^*(\mathrm{Sp(n)};\mathbb{Z})= \bigwedge (y_3,y_7,\cdots, y_{4n-1}),\\
& H^*(\mathrm{Sp(\infty)};\mathbb{Z})= \bigwedge (y_3,y_7,\cdots),\\
&H^*(\mathrm{BSp(\infty)};\mathbb{Z})= \mathbb{Z}[q_1,q_2,\cdots],
\end{align*}
where $y_{4i-1}=\sigma q_i$, $\sigma$ is the cohomology suspension and $q_i$ is the $i-$th symplectic Pontrjagin
class. Consider the projection map $\pi\colon \mathrm{Sp(\infty)}\rightarrow X_n$. Also, the cohomologies of $X_n$ and
$\Omega X_n$ as an algebra are given by
\begin{align*}
& H^*(X_n;\mathbb{Z})= \bigwedge (\bar{y}_{4n+3}, \bar{y}_{4n+7}, \cdots), \\
& H^*(\Omega X_n;\mathbb{Z})= \mathbb{Z}\{b_{4n+2}, b_{4n+6}, \cdots, b_{8n+2} \} \quad (*\leq 8n+2 ),
\end{align*}
where $\pi^*(\bar{y}_{4i+3})=y_{4i+3}$ and $b_{4n+4j-2}=\sigma (\bar{y}_{4n+4j-1})$. Consider the following fibre
sequence
\begin{eqnarray}
\Omega \mathrm{Sp(\infty)}\overset{\Omega \pi}\longrightarrow \Omega X_n\overset{\delta}\longrightarrow \mathrm{Sp(n)}
\overset{j}\longrightarrow \mathrm{Sp(\infty)}\overset{\pi}\longrightarrow X_n.
\end{eqnarray}
Now applying the functor $[X,\quad]$ to the fibration sequence $(3.1)$, there is an exact sequence
\begin{eqnarray}
[X,\Omega \mathrm{Sp(\infty)}]\overset{(\Omega \pi)_*}\longrightarrow [X,\Omega X_n]\overset{\delta_*}\longrightarrow
\mathrm{Sp_n(X)}\overset{j_*}\longrightarrow [X,\mathrm{Sp(\infty)}]\overset{\pi_*}\longrightarrow [X,X_n].
\end{eqnarray}
Note that $X_n$ and $\Omega X_n$ have the cellular structure as follows
\begin{align*}
& X_n\simeq S^{4n+3}\cup e^{4n+7}\cup e^{4n+11}\cup \cdots, \\
& \Omega X_n\simeq S^{4n+2}\cup e^{4n+6}\cup e^{4n+10}\cup \cdots.
\end{align*}
According to the $\rm{CW}$-structure of $X_n$, we have the following isomorphisms
\begin{equation*}
\pi_i(X_n)=0 \quad (\rm{for} \hspace{2mm} i\leq 4n+2), \qquad \pi_{4n+3}(X_n)\cong \mathbb{Z}.
\end{equation*}
We use the isomorphism ${\widetilde{KSp}}^{-i}(X)\cong [\Sigma^iX,\mathrm{BSp(\infty)}]$. Observe that
$$[X, \mathrm{Sp(\infty)}]\cong [\Sigma X, \mathrm{BSp(\infty)}]\cong {\widetilde{KSp}}^{-1}(X).$$
Apply ${\widetilde{KSp}}^{-1}$ to the cofibration sequence $S^{4m-1}\wedge Q_{n-m}\rightarrow X\rightarrow S^{4n+2}$.
We know that ${\widetilde{KSp}}^{-1}(S^{4n+2})=0$ for every $n\geq 1$, so by use of induction, we can conclude that
${\widetilde{KSp}}^{-1}(X)=0$. On the other hand, we know that $\Omega X_n$ is $(4n+1)$-connected and $H^{4n+2}(\Omega
X_n)\cong \mathbb{Z}$, which is generated by $b_{4n+2}=\sigma(\bar{y}_{4n+3})$. The map $b_{4n+2}\colon\Omega X_n
\rightarrow K(\mathbb{Z},4n+2)$ is a loop map and is a $(4n+3)$-equivalence. Since dim$X\leq 4n+2$, the map
$(b_{4n+2})_*\colon[X,\Omega X_n] \rightarrow H^{4n+2}(X)$ is an isomorphism of groups. Thus, we rewrite the exact
sequence $(3.2)$ as the following exact sequence
\begin{eqnarray*}
{\widetilde{KSp}}^{-2}(X)\overset{\vartheta}\longrightarrow H^{4n+2}(X)\rightarrow \mathrm{Sp_n(X)}\rightarrow 0.
\end{eqnarray*}
Therefore we get the following lemma.
\begin{lem}\label{c1}
\rm{There is an isomorphism $\mathrm{Sp}_n(X)\cong \mathrm{Coker} \hspace{1mm} \vartheta$.} $\quad \Box$
\end{lem}
In what follows in this section, we will calculate the image of the map $\vartheta$.\par
Let $Y$ be a $CW$-complex with dim $Y \leq 4n+ 2$, we denote $[Y, \mathrm{U(2n+1)}]$ by $\mathrm{U_{2n+1}(Y)}$. By
[Theorem 1.1, 11], there is an exact sequence
\begin{equation*}
\tilde{K}^{-2}(Y)\overset{\varphi}\longrightarrow H^{4n+2}(Y)\rightarrow \mathrm{U_{2n+1}(Y)}\rightarrow
\tilde{K}^{-1}(Y)\rightarrow 0,
\end{equation*}
where for any $f\in \tilde{K}^{-2}(Y)$, the map $\varphi$ is defined as follows:
$$\varphi(f)=(2n+1)!\mathrm{ch}_{2n+1}(f),$$
where $\mathrm{ch}_i$ denotes the $2i$-dimensional part of the Chern character. Also, we use the isomorphism
$\tilde{K}^{-i}(Y)\cong [\Sigma^iY,\mathrm{BU(\infty)}]$. In this paper, we denote both the canonical inclusion
$\mathrm{Sp(n)}\hookrightarrow \mathrm{U(2n)}$ and the induced map ${\widetilde{KSp}}^{*}(X) \rightarrow
{\tilde{K}}^{*}(X)$ by $c'$. By [Theorem 1.3, 25], there is the following commutative diagram
\begin{equation}
\begin{tikzpicture}[baseline=(current bounding box.center)]
\matrix(m)[matrix of math nodes,
 row sep=2em, column sep=2em,
 text height=1.5ex, text depth=0.20ex]
   {{\widetilde{KSp}}^{-2}(X)& H^{4n+2}(X)\\ {\tilde{K}}^{-2}(X)& H^{4n+2}(X).\\}; \path[->,font=\scriptsize]
  (m-1-1) edge node[above] {$\vartheta$}(m-1-2)
  (m-1-1) edge node[left] {$c'$} (m-2-1)
  (m-1-2) edge node[right] {$(-1)^{n+1}$}(m-2-2)
      (m-2-1) edge node[above] {$\varphi$} (m-2-2);
 \end{tikzpicture}
 \end{equation}
Therefore, to calculate the image of $\vartheta$, we first calculate the image of $\varphi$. We denote the free
abelian group with a basis $e_1,e_2,\cdots$, by $\mathbb{Z}\{ e_1,e_2,\cdots \}$ and the direct sum of $k$ copies of
$\mathbb{Z}$ by $\mathbb{Z}^k$. We need the following lemmas.
\begin{lem}\label{c'1}
\rm{ The group ${\widetilde{KSp}}^{-2}(\Sigma^{4m}Q_{n-m})$ is isomorphic to a finite group.}
\end{lem}
\begin{proof}
To prove this lemma, we will use induction. Note that we have the following isomorphisms
\begin{align*}
&{\widetilde{KSp}}^{-2}(S^{4i+2})\cong \mathbb{Z}, \quad &&{\widetilde{KSp}}^{-2}(S^{4n+1})\cong 0,\\
&{\widetilde{KSp}}^{-2}(S^{4n})\cong \left\{ \begin{array}{ll}
0 & \textrm{if $n$ is even },\\
\mathbb{Z}/2\mathbb{Z} & \textrm{if $n$ is odd},\\
\end{array} \right.\quad &&{\widetilde{KSp}}^{-2}(S^{4n-1})\cong \left\{ \begin{array}{ll}
\mathbb{Z}/2\mathbb{Z} & \textrm{if $n$ is even },\\
0 & \textrm{if $n$ is odd}.\\
\end{array} \right.
\end{align*}
First, for case $Q_2$,  we calculate exactly the group ${\widetilde{KSp}}^{-2}(\Sigma^{4m}Q_2)$. Consider the homotopy
cofibration sequence
\begin{eqnarray}
S^{4m+6}\overset{\Sigma^{4m}v'}\longrightarrow S^{4m+3}\rightarrow \Sigma^{4m} Q_2\rightarrow
S^{4m+7}\overset{\Sigma^{4m+1}v'}\longrightarrow S^{4m+4},
\end{eqnarray}
where $v'$ is a generator of $\pi_6(S^3)\cong \mathbb{Z}_{12}$. By applying ${\widetilde{KSp}}^{-2}$ to the
cofibration $(3.4)$, we get the exact sequence
\begin{align*}
{\widetilde{KSp}}^{-2}(S^{4m+4})\overset{(\Sigma^{4m+1}v')^*}\longrightarrow
{\widetilde{KSp}}^{-2}(S^{4m+7})\rightarrow {\widetilde{KSp}}^{-2}(\Sigma^{4m} Q_2)&\rightarrow
{\widetilde{KSp}}^{-2}(S^{4m+3})\\&\rightarrow{\widetilde{KSp}}^{-2}(S^{4m+6}).
\end{align*}
Now, in cases where $m$ is even and $m$ is odd, we obtain the exact sequences
$$\mathbb{Z}_2\overset{(\Sigma^{4m+1}v')^*}\longrightarrow \mathbb{Z}_2 \rightarrow
{\widetilde{KSp}}^{-2}(\Sigma^{4m} Q_2) \rightarrow 0,\quad
0 \rightarrow  {\widetilde{KSp}}^{-2}(\Sigma^{4m} Q_2) \rightarrow \mathbb{Z}_2\rightarrow \mathbb{Z},$$
respectively. Since $\mathbb{Z}$ is torsion-free, any map $\mathbb{Z}_2\rightarrow \mathbb{Z}$ is zero map, thus in
case $m$ is odd, by exactness we obtain the group ${\widetilde{KSp}}^{-2}(\Sigma^{4m} Q_2)$ is isomorphic to
$\mathbb{Z}_2$. Let $m$ be even. Since $\mathrm{Sp(\infty)}$ is homotopy euivalent to $\Omega^4\mathrm{O(\infty)}$, we
can determine the map ${\widetilde{KSp}}^{-2}(S^{4m+4})\overset{(\Sigma^{4m+1}v')^*}\longrightarrow
{\widetilde{KSp}}^{-2}(S^{4m+7})$ by $(\Sigma^{4m+1}v')^*\colon \pi_{4m+9}(\mathrm{SO(\infty)})\rightarrow
\pi_{4m+12}(\mathrm{SO(\infty)})$. Let $t_1$ be a generator of $\pi_{4m+9}(\mathrm{SO(\infty)})\cong \mathbb{Z}_2$.
Then $(\Sigma^{4m+1}v')^* (t_1)$ is the composite $S^{4m+12}\overset{\Sigma^{4m+6}v'} \longrightarrow
S^{4m+9}\overset{t_1} \longrightarrow \mathrm{SO(\infty)}$. Note that $t_1\circ \Sigma^{4m+6}v'\simeq t_1\circ
2v_{4m+9}\simeq 2t_1\circ v_{4m+9}$ is null homotopic, where by \cite{[T]} we have $\Sigma^{4m+6}v'\simeq 2v_{4m+9}$,
stably. Thus the map $(\Sigma^{4m+1}v')^*$ is the zero map. Therefore we can conclude that the group
${\widetilde{KSp}}^{-2}(\Sigma^{4m} Q_2)$ is also isomorphic to $\mathbb{Z}_2$.\par
Now, by induction, we let the Lemma hold for the case $Q_{n-m-1}$, that is, the group
${\widetilde{KSp}}^{-2}(\Sigma^{4m}Q_{n-m-1})$ is isomorphic to a finite group. Consider the cofibration sequence
$$S^{4n-2}\overset{\theta}\longrightarrow \Sigma^{4m}Q_{n-m-1}\rightarrow \Sigma^{4m} Q_{n-m}\rightarrow S^{4n-1}
\overset{\Sigma\theta}\longrightarrow \Sigma^{4m+1} Q_{n-m-1}.$$
This sequence induces an exact sequence
\begin{align*}
{\widetilde{KSp}}^{-2}(\Sigma^{4m+1} Q_{n-m-1})\overset{(\Sigma\theta)^*}\longrightarrow
{\widetilde{KSp}}^{-2}(S^{4n-1})\rightarrow {\widetilde{KSp}}^{-2}(\Sigma^{4m} Q_{n-m})&\rightarrow
{\widetilde{KSp}}^{-2}(\Sigma^{4m} Q_{n-m-1})\\& \overset{\theta^*}\rightarrow {\widetilde{KSp}}^{-2}(S^{4n-2}).
\end{align*}
Now, in cases where $n$ is even and $n$ is odd, we get the exact sequences
$$\overset{(\Sigma\theta)^*}\longrightarrow \mathbb{Z}_2\rightarrow {\widetilde{KSp}}^{-2}(\Sigma^{4m}
Q_{n-m})\rightarrow {\widetilde{KSp}}^{-2}(\Sigma^{4m} Q_{n-m-1}) \overset{\theta^*}\rightarrow \mathbb{Z},$$
$$0\rightarrow {\widetilde{KSp}}^{-2}(\Sigma^{4m} Q_{n-m})\rightarrow {\widetilde{KSp}}^{-2}(\Sigma^{4m} Q_{n-m-1})
\overset{{\theta'}^*}\rightarrow \mathbb{Z},$$
respectively, where by the assumption of induction, the group ${\widetilde{KSp}}^{-2}(\Sigma^{4m} Q_{n-m-1})$ is
isomorphic to a finite group. Therefore in both cases by exactness, we can conclude that the group
${\widetilde{KSp}}^{-2}(\Sigma^{4m} Q_{n-m})$ is also isomorphic to a finite group.
\end{proof}
\begin{lem}\label{c2}
\rm{The following hold:
\begin{enumerate}
  \item There is an isomorphism ${\widetilde{KSp}}^{-2}(X)\rightarrow {\widetilde{KSp}}^{-2}(S^{4m+2})\oplus \cdots
\oplus  {\widetilde{KSp}}^{-2}(S^{4n+2})\cong \mathbb{Z}^{n-m+1}$ such that
${\widetilde{KSp}}^{-2}(X)=\mathbb{Z}\{\xi_1,\xi_2,\cdots ,\xi_{n-m+1}\}$, where $\xi_1\in
{\widetilde{KSp}}^{-2}(S^{4m+2})$, $\xi_2$ $\in$ ${\widetilde{KSp}}^{-2}$ $(S^{4m+6})$, $\cdots$, and also
$\xi_{n-m+1}\in {\widetilde{KSp}}^{-2}(S^{4n+2})$,
  \item There is an isomorphism ${\tilde{K}}^{-2}(X)\rightarrow {\tilde{K}}^{-2}(S^{4m+2}) \oplus \cdots \oplus
{\tilde{K}}^{-2}(S^{4n+2}) \cong \mathbb{Z}^{n-m+1}$ such that ${\tilde{K}}^{-2}(X)=\mathbb{Z}\{\xi'_1,\xi'_2,\cdots
,\xi'_{n-m+1}\}$, where $\xi'_1\in {\tilde{K}}^{-2}(S^{4m+2})$, $\xi'_2\in {\tilde{K}}^{-2}(S^{4m+6})$, $\cdots$,
and also $\xi'_{n-m+1}\in {\tilde{K}}^{-2}(S^{4n+2})$,
  \item If $m$ is an even integer, then
\begin{displaymath}
c'(\xi_1)=\xi'_1,\quad c'(\xi_2)=2\xi'_2,\quad \cdots, \quad c'(\xi_{n-m+1}) = \left\{ \begin{array}{ll}
\xi'_{n-m+1} & \textrm{if $n$ is even },\\
2\xi'_{n-m+1} & \textrm{if $n$ is odd },\\
\end{array} \right.
\end{displaymath}
and if $m$ is an odd integer, then
\begin{displaymath}
c'(\xi_1)=2\xi'_1,\quad c'(\xi_2)=\xi'_2,\quad \cdots, \quad c'(\xi_{n-m+1}) = \left\{ \begin{array}{ll}
\xi'_{n-m+1} & \textrm{if $n$ is even },\\
2\xi'_{n-m+1} & \textrm{if $n$ is odd }.\\
\end{array} \right.
\end{displaymath}
\end{enumerate}
}
\end{lem}
\begin{proof}
First, we consider the case $Q_2$. Put $L_1={\widetilde{KSp}}^{-2}(S^{4m+3})$ and
$L_2={\widetilde{KSp}}^{-2}(S^{4m+5})$. The cofibration sequence
$$S^{4m+2}\rightarrow S^{4m-1}\wedge Q_2\rightarrow S^{4m+6}$$
induces the following commutative diagram of exact sequences
\begin{equation*}
\begin{tikzpicture}[baseline=(current bounding box.center)]
\matrix(m)[matrix of math nodes,
 row sep=2em, column sep=2em,
 text height=1.5ex, text depth=0.20ex]
  {L_1&{\widetilde{KSp}}^{-2}(S^{4m+6})&{\widetilde{KSp}}^{-2}(S^{4m-1}\wedge
  Q_2)&{\widetilde{KSp}}^{-2}(S^{{4m+2}})&L_2\\0&{\tilde{K}}^{-2}(S^{4m+6})&{\tilde{K}}^{-2}(S^{4m-1}\wedge
  Q_2)&{\tilde{K}}^{-2}(S^{{4m+2}})&0,\\}; \path[->,font=\scriptsize]
  (m-1-1) edge (m-1-2)
  (m-1-2) edge (m-1-3)
  (m-1-3) edge (m-1-4)
  (m-1-4) edge (m-1-5)
  (m-1-2) edge node[left] {$c'=\sigma_j$}(m-2-2)
  (m-1-3) edge node[left] {$c'$}(m-2-3)
  (m-1-4) edge node[left] {$c'=\sigma_i$}(m-2-4)
      (m-2-1) edge (m-2-2)
  (m-2-2) edge (m-2-3)
  (m-2-3) edge (m-2-4)
  (m-2-4) edge (m-2-5);
 \end{tikzpicture}
 \end{equation*}
where $\sigma_i=1, \sigma_j=2 $ when $m$ is even and $\sigma_i=2, \sigma_j=1$ when $m$ is odd (for example, see
\cite{[KK]}). Note that there is an isomorphism ${\tilde{K}}^{-2}(S^{2i})\cong\mathbb{Z}$. We know that $L_1$ is
isomorphic to zero if $m$ is even and is isomorphic to $\mathbb{Z}_2$ if $m$ is odd. Also, we have that $L_2$ is zero.
So, when $m$ is odd, then we get the following exact sequence
$$\mathbb{Z}_2 \rightarrow \mathbb{Z} \rightarrow {\widetilde{KSp}}^{-2}(S^{4m-1}\wedge Q_2)\rightarrow \mathbb{Z}
\rightarrow 0.$$
Thus, by the exactness we get ${\widetilde{KSp}}^{-2}(S^{4m-1}\wedge Q_2)$ is isomorphic to $\mathbb{Z}\oplus
\mathbb{Z}$. When $m$ is even, then we get the following exact sequence
$$0\rightarrow \mathbb{Z} \rightarrow {\widetilde{KSp}}^{-2}(S^{4m-1}\wedge Q_2)\rightarrow \mathbb{Z} \rightarrow
0.$$
Thus, this exact sequence splits, and we get ${\widetilde{KSp}}^{-2}(S^{4m-1}\wedge Q_2)$ is isomorphic to
$\mathbb{Z}\oplus \mathbb{Z}$. Note that in both cases ${\tilde{K}}^{-2}(S^{4m-1}\wedge Q_2)$ is a free abelian group
isomorphic to $\mathbb{Z}\oplus \mathbb{Z}$. Therefore we have
$${\widetilde{KSp}}^{-2}(S^{4m-1}\wedge Q_2)=\mathbb{Z}\{\xi_1,\xi_2\}, \quad {\tilde{K}}^{-2}(S^{4m-1}\wedge
Q_2)=\mathbb{Z}\{\xi'_1,\xi'_2\},$$
where $\xi_1\in {\widetilde{KSp}}^{-2}(S^{4m+2})$, $\xi_2\in {\widetilde{KSp}}^{-2}(S^{4m+6})$, $\xi'_1\in
{\tilde{K}}^{-2}(S^{4m+2})$ and $\xi'_2\in {\tilde{K}}^{-2}(S^{4m+6})$.\par
Put $L_3={\widetilde{KSp}}^{-2}(S^{4m}\wedge Q_{n-m})$. There is a cofibration sequence
$$S^{4n+1}\overset{\kappa}\rightarrow S^{4m-1}\wedge Q_{n-m} \rightarrow X\rightarrow
S^{4n+2}\overset{\Sigma\kappa}\rightarrow S^{4m} \wedge Q_{n-m}.$$
This sequence induces the following commutative diagram of exact sequences
\begin{equation*}
\begin{tikzpicture}[baseline=(current bounding box.center)]
\matrix(m)[matrix of math nodes,
 row sep=2em, column sep=2em,
 text height=1.5ex, text depth=0.20ex]
  {L_3&{\widetilde{KSp}}^{-2}(S^{4n+2})&{\widetilde{KSp}}^{-2}(X)&{\widetilde{KSp}}^{-2}(S^{4m-1}\wedge
  Q_{n-m})&0\\0&{\tilde{K}}^{-2}(S^{4n+2})&{\tilde{K}}^{-2}(X)&{\tilde{K}}^{-2}(S^{4m-1}\wedge Q_{n-m})&0, \quad
  (\ast)\\}; \path[->,font=\scriptsize]
  (m-1-2) edge (m-1-3)
  (m-1-3) edge (m-1-4)
  (m-1-4) edge (m-1-5)
  (m-1-1) edge node[above] {$(\Sigma\kappa)^*$} (m-1-2)
  (m-1-2) edge node[left] {$c'=\sigma_k$}(m-2-2)
  (m-1-3) edge node[left] {$c'$}(m-2-3)
  (m-1-4) edge node[left] {$c'$}(m-2-4)
      (m-2-1) edge (m-2-2)
  (m-2-2) edge (m-2-3)
  (m-2-3) edge (m-2-4)
  (m-2-4) edge (m-2-5);
 \end{tikzpicture}
 \end{equation*}
where $\sigma_k=1$ when $n$ is even and $\sigma_k=2$ when $n$ is odd. By Lemma \ref{c'1}, the group $L_3$ is
isomorphic to a finite group, so the map $(\Sigma\kappa)^*$ is trivial. Thus, we can rewrite the upper row of the
exact sequence in diagram $(\ast)$ as the following exact sequence
\begin{align*}
0\rightarrow {\widetilde{KSp}}^{-2}(S^{4n+2})\rightarrow {\widetilde{KSp}}^{-2}(X)\rightarrow
{\widetilde{KSp}}^{-2}(S^{4m-1}\wedge Q_{n-m})\rightarrow 0.
\end{align*}
Therefore, the upper and lower rows of the exact sequences in the diagram $(\ast)$ split. Thus by induction, we can
conclude that ${\widetilde{KSp}}^{-2}(X)$ and ${\tilde{K}}^{-2}(X)$ are free abelian groups that are isomorphic to
$\mathbb{Z}^{n-m+1}$ with their basis $\xi_1,\xi_2,\cdots ,\xi_{n-m+1}$ and $\xi'_1,\xi'_2,\cdots ,\xi'_{n-m+1}$,
respectively, where $\xi_{n-m+1}\in {\widetilde{KSp}}^{-2}(S^{4n+2})$ and $\xi'_{n-m+1}\in
{\tilde{K}}^{-2}(S^{4n+2})$. Now, according to the definition of the maps $c'$, we can choose $\xi_1$, $\xi'_1$,
$\xi_2$, $\xi'_2$, $\cdots$, $\xi_{n-m+1}$ and $\xi'_{n-m+1}$ such that if $m$ is an even integer, then we have
\begin{displaymath}
c'(\xi_1)=\xi'_1,\quad c'(\xi_2)=2\xi'_2,\quad \cdots, \quad c'(\xi_{n-m+1}) = \left\{ \begin{array}{ll}
\xi'_{n-m+1} & \textrm{if $n$ is even },\\
2\xi'_{n-m+1} & \textrm{if $n$ is odd },\\
\end{array} \right.
\end{displaymath}
and if $m$ is an odd integer, then we have
\begin{displaymath}
c'(\xi_1)=2\xi'_1,\quad c'(\xi_2)=\xi'_2,\quad \cdots, \quad c'(\xi_{n-m+1}) = \left\{ \begin{array}{ll}
\xi'_{n-m+1} & \textrm{if $n$ is even },\\
2\xi'_{n-m+1} & \textrm{if $n$ is odd }.
\end{array} \right.
\end{displaymath}
\end{proof}
Consider the map $c'\colon \mathrm{Sp(n-m+1)}\rightarrow \mathrm{SU(2n-2m+2)}$. By restriction of the map $c'$ to
their quasi-projective spaces, we get to map ${\bar{c}}'\colon Q_{n-m+1}\rightarrow \Sigma\mathbb{C}P^{2n-2m+1}$. The
cohomologies of $Q_{n-m+1}$ and $\Sigma\mathbb{C}P^{2n-2m+1}$ are given by
\begin{align*}
&H^* (Q_{n-m+1})=\mathbb{Z}\{{\bar{y}}_3,{\bar{y}}_7, \cdots ,{\bar{y}}_{4n-4m+3}\},\\
&H^* (\Sigma\mathbb{C}P^{2n-2m+1})=\mathbb{Z}\{{\bar{x}}_3,{\bar{x}}_5, {\bar{x}}_7,\cdots, {\bar{x}}_{4n-4m+3}\},
\end{align*}
such that ${\bar{c}}'({\bar{x}}_3)={\bar{y}}_3$, ${\bar{c}}'({\bar{x}}_5)=0$, $ {\bar{c}}'({\bar{x}}_7)={\bar{y}}_7$,
$\cdots$, and ${\bar{c}}'({\bar{x}}_{4n-4m+3})={\bar{y}}_{4n-4m+3}$. We denote a generator of $\tilde{K}(S^{2n})\cong
\mathbb{Z}$ by $\zeta_n$. Recall that
$$H^*(\mathbb{C}P^{2n-2m+1})=\mathbb{Z}[t]/(t^{2n-2m+2}),\qquad
K(\mathbb{C}P^{2n-2m+1})=\mathbb{Z}[x]/(x^{2n-2m+2}),$$
and $\mathrm{ch}(x)=t+\frac{1}{2}t^2+ \cdots+\frac{1}{(2n-2m+1)!}t^{2n-2m+1}$, where $|t|=2$. Note that
${\tilde{K}}^{-2}(\Sigma^{4m}\mathbb{C}P^{2n-2m+1})\cong
{\tilde{K}}^{0}(\Sigma^{4m+2}\mathbb{C}P^{2n-2m+1})$ is a free abelian group generated by $\zeta_{2m+1}\otimes x$,
$\zeta_{2m+1}\otimes x^2$, $\cdots$, $\zeta_{2m+1}\otimes x^{2n-2m+1}$, with the following Chern characters
\begin{align*}
&\mathrm{ch}_{2n+1}(\zeta_{2m}\otimes
x)=\mathrm{ch}_{2m}(\zeta_{2m})\mathrm{ch}_{2n-2m+1}(x)=\frac{1}{(2n-2m+1)!}\sigma^{4m}t^{2n-2m+1},\\\\
&\mathrm{ch}_{2n+1}(\zeta_{2m}\otimes x^2)=\mathrm{ch}_{2m}(\zeta_{2m})\mathrm{ch}_{2n-2m+1}(x^2)=A\sigma^{4m},\\\\
&\mathrm{ch}_{2n+1}(\zeta_{2m}\otimes x^3)=\mathrm{ch}_{2m}(\zeta_{2m})\mathrm{ch}_{2n-2m+1}(x^3)=B\sigma^{4m},\\
\vdots\\
&\mathrm{ch}_{2n+1}(\zeta_{2m}\otimes
x^{2n-2m+1})=\mathrm{ch}_{2m}(\zeta_{2m})\mathrm{ch}_{2n-2m+1}(x^{2n-2m+1})=C\sigma^{4m},
\end{align*}
where $A$, $B$, and $C$ are as follows
$$A= \mathrm{ch}_{2n-2m+1}(x^2)= \sum_{\substack {i_1+i_2=2n-2m+1,\\ 1\leq i_1\leq n-m}}\mathrm{ch}_{i_1}x
\mathrm{ch}_{i_2}x =\sum_{k=1}^{n-m}\frac{1}{k!(2n-2m+1-k)!}t^{2n-2m+1},$$
\begin{align*}
B= \mathrm{ch}_{2n-2m+1}(x^3)=&\mathrm{ch}_1x\sum_{\substack {i_1+i_2=2n-2m,\\ 1\leq i_1\leq i_2}}\mathrm{ch}_{i_1}x
\mathrm{ch}_{i_2}x+\mathrm{ch}_2 x\sum_{\substack {i_1+i_2=2n-2m-1,\\ 2\leq i_1\leq i_2}}\mathrm{ch}_{i_1}x
\mathrm{ch}_{i_2}x+\cdots \\
&+ \mathrm{ch}_{n-m-1}x\sum_{\substack {i_1+i_2=n-m+2,\\ n-m-1\leq i_1\leq i_2}}\mathrm{ch}_{i_1}x \mathrm{ch}_{i_2}x
+ \sum_{\substack {i_1+i_2=2n-2m+1, \\ i_1, i_2\geq 1}}\mathrm{ch}_{i_1}x \mathrm{ch}_{i_2}x^2,
\end{align*}
and
\begin{align*}
C= \mathrm{ch}_{2n-2m+1}(x^{2n-2m+1})=&\mathrm{ch}_1x\sum_{\substack {i_1+\cdots + i_{2n-2m}=2n-2m,\\ 1\leq i_1\leq
i_2 \leq \cdots \leq
i_{2n-2m}}} \mathrm{ch}_{i_1}x^{i_1}\cdots \mathrm{ch}_{i_{2n-2m}}x^{i_{2n-2m}} \\&+\mathrm{ch}_2x^2\sum_{\substack
{i_1+\cdots +
i_{n-m-1}=2n-2m-1,\\ 2\leq i_1\leq i_2\leq \cdots \leq i_{n-m-1}}}\mathrm{ch}_{i_1}x^{i_1} \cdots
\mathrm{ch}_{i_{n-m-1}}x^{i_{n-m-1}}
\\
&+\mathrm{ch}_3x^3 \sum_{\substack {i_1+\cdots +i_{n-m-2}=2n-2m-2,\\ 3\leq i_1\leq i_2\leq \cdots \leq
i_{n-m-2}}}\mathrm{ch}_{i_1}x^{i_1}\cdots  \mathrm{ch}_{i_{n-m-2}}x^{i_{n-m-2}} +\cdots \\&+
\mathrm{ch}_{n-m}x^{n-m}\sum_{\substack
{i_1=n-m+1,\\ i_1\geq n-m}}\mathrm{ch}_{i_1}x^{i_1}.
\end{align*}
According to the map ${\bar{c}}'\colon {\tilde{K}}^{-2}(\Sigma^{4m}\mathbb{C}P^{2n-2m+1})\rightarrow
{\tilde{K}}^{-2}(S^{4m-1}\wedge Q_{n-m+1})$, we can put $\xi'_1$, $\xi'_2$, $\cdots$ and  $\xi'_{n-m+1}$ such that
$\xi'_1={\bar{c}}'(\zeta_{2m}\otimes x)$, $\xi'_2={\bar{c}}'(\zeta_{2m}\otimes x^3)$, $\cdots$, and
$\xi'_{n-m+1}={\bar{c}}'(\zeta_{2m}\otimes x^{2n-2m+1})$. We have the following proposition.
\begin{prop}\label{c3}
\rm{The image of the map ${\widetilde{KSp}}^{-2}(X)\overset{\vartheta}\longrightarrow
H^{4n+2}(X)$ is generated by
\begin{displaymath}
\left\{ \begin{array}{ll}
\frac{(2n+1)!}{(2n-2m+1)!} \sigma^{4m-1}\otimes \bar{y}_{4n-4m+3} &  \textrm{if $m$ is even },\\\\
\frac{2(2n+1)!}{(2n-2m+1)!} \sigma^{4m-1}\otimes \bar{y}_{4n-4m+3} &  \textrm{if $m$ is odd }.\\
\end{array} \right.
\end{displaymath}
}
\end{prop}
\begin{proof}
Consider the following commutative diagram
\begin{equation}
\begin{tikzpicture}[baseline=(current bounding box.center)]
\matrix(m)[matrix of math nodes,
 row sep=2em, column sep=2em,
 text height=1.5ex, text depth=0.20ex]
 {{\tilde{K}}^{-2}(\Sigma^{4m}\mathbb{C}P^{2n-2m+1})& H^{4n+2}(\Sigma^{4m}\mathbb{C}P^{2n-2m+1})\\
 {\tilde{K}}^{-2}(S^{4m-1}\wedge Q_{n-m+1})& H^{4n+2}(S^{4m-1}\wedge Q_{n-m+1}),\\}; \path[->,font=\scriptsize]
  (m-1-1) edge node[above] {$\varphi'$}(m-1-2)
  (m-1-1) edge node[left] {${\bar{c}}'$} (m-2-1)
  (m-1-2) edge node[right] {$\cong$}(m-2-2)
      (m-2-1) edge node[above] {$\varphi$} (m-2-2);
 \end{tikzpicture}
\end{equation}
where the map $\varphi'$ is defined similarly to the map $\varphi$. That is, for any $f\in
{\tilde{K}}^{-2}(\Sigma^{4m}\mathbb{C}P^{2n-2m+1})$,  $\varphi'(f)=(2n+1)!\mathrm{ch}_{2n+1}(f)$. By definition of the
map $\varphi'$, we have
\begin{align*}
&\varphi'(\zeta_{2m}\otimes x)=(2n+1)!\mathrm{ch}_{2n+1}(\zeta_{2m}\otimes x)=
\frac{(2n+1)!}{(2n-2m+1)!}\sigma^{4m}t^{2n-2m+1},\\\\
&\varphi'(\zeta_{2m}\otimes x^2)=(2n+1)!\mathrm{ch}_{2n+1}(\zeta_{2m}\otimes x^2)=
A(2n+1)!\sigma^{4m},\\\\
&\varphi'(\zeta_{2m}\otimes x^3)=(2n+1)!\mathrm{ch}_{2n+1}(\zeta_{2m}\otimes x^3)=
B(2n+1)!\sigma^{4m},\\
\vdots\\
&\varphi'(\zeta_{2m}\otimes x^{2n-2m+1})=(2n+1)!\mathrm{ch}_{2n+1}(\zeta_{2m}\otimes x^{2n-2m+1})=
C(2n+1)!\sigma^{4m}.
\end{align*}
Therefore, according to the commutativity of diagram $(3.5)$, we have
\begin{align*}
&\varphi(\xi'_1)=\varphi'(\zeta_{2m}\otimes x)=
\frac{(2n+1)!}{(2n-2m+1)!}\sigma^{4m-1}\otimes \bar{y}_{4n-4m+3},\\\\
&\varphi(\xi'_2)=\varphi'(\zeta_{2m}\otimes x^3)=B'(2n+1)!\sigma^{4m-1}\otimes \bar{y}_{4n-4m+3},\\
\vdots\\
&\varphi(\xi'_{n-m+1})=\varphi'(\zeta_{2m}\otimes x^{2n-2m+1})=C'(2n+1)!\sigma^{4m-1}\otimes \bar{y}_{4n-4m+3},
\end{align*}
where $B=B't^{2n-2m+1}$ and $C=C't^{2n-2m+1}$. Thus, by the commutativity of diagram $(3.3)$, we get
\begin{displaymath}
\vartheta(\xi_1)= \varphi(c'(\xi_1))=\pm \left\{ \begin{array}{ll}
\frac{(2n+1)!}{(2n-2m+1)!} \sigma^{4m-1}\otimes \bar{y}_{4n-4m+3} & \textrm{if $m$ is even },\\\\
\frac{2(2n+1)!}{(2n-2m+1)!} \sigma^{4m-1}\otimes \bar{y}_{4n-4m+3} & \textrm{if $m$ is odd },\\
\end{array} \right.
\end{displaymath}
\begin{displaymath}
\vartheta(\xi_2)= \varphi(c'(\xi_2))=\pm \left\{ \begin{array}{ll}
2B'(2n+1)! \sigma^{4m-1}\otimes \bar{y}_{4n-4m+3} & \textrm{if $m$ is even },\\\\
B'(2n+1)! \sigma^{4m-1}\otimes \bar{y}_{4n-4m+3} & \textrm{if $m$ is odd },\\
\end{array} \right.
\end{displaymath}
$\qquad \vdots$
\begin{displaymath}
\vartheta(\xi_{n-m+1})= \varphi(c'(\xi_{n-m+1}))= \left\{ \begin{array}{ll}
-(2n+1)! C' \sigma^{4m-1}\otimes \bar{y}_{4n-4m+3} & \textrm{if $m$ is even },\\\\
 2(2n+1)! C' \sigma^{4m-1}\otimes \bar{y}_{4n-4m+3} & \textrm{if $m$ is odd }.\\
\end{array} \right.
\end{displaymath}
Since the coefficients $B',\cdots, C'$ are divisible by $\frac{2}{(2n-2m+1)!}$, the image of $\vartheta$ is included
in the submodule generated by $\frac{(2n+1)!}{(2n-2m+1)!} \sigma^{4m-1}\otimes \bar{y}_{4n-4m+3}$ if $m$ is even and
is included in the submodule generated by $\frac{2(2n+1)!}{(2n-2m+1)!} \sigma^{4m-1}\otimes \bar{y}_{4n-4m+3}$ if $m$
is
odd. Thus, we can conclude that
\begin{displaymath}
\mathrm{Im}\hspace{1mm} \vartheta\cong \left\{ \begin{array}{ll}
\mathbb{Z}\{\frac{(2n+1)!}{(2n-2m+1)!} \sigma^{4m-1}\otimes \bar{y}_{4n-4m+3}\} & \textrm{if $m$ is even },\\\\
\mathbb{Z}\{\frac{2(2n+1)!}{(2n-2m+1)!} \sigma^{4m-1}\otimes \bar{y}_{4n-4m+3}\} & \textrm{if $m$ is odd }.
\end{array} \right.
\end{displaymath}
\end{proof}
Therefore by Lemma \ref{c1} and Proposition \ref{c3}, we get the following theorem.
\begin{thm}\label{c4}
\rm{ There is an isomorphism
\begin{displaymath}
[X, \mathrm{Sp(n)}]\cong \left\{ \begin{array}{ll}
\mathbb{Z}_{\frac{(2n+1)!}{(2n-2m+1)!}} & \textrm{if $m$ is even },\\\\
\mathbb{Z}_{\frac{2(2n+1)!}{(2n-2m+1)!}} & \textrm{if $m$ is odd }. \quad \Box
\end{array} \right.
\end{displaymath}
}
\end{thm}
\section{Proof of Theorem \ref{a1}}\label{d0}
In this section, we will prove Theorem \ref{a1}. We recall that $X=S^{4m-1}\wedge Q_{n-m+1}$. In the following
theorem, which was proved in \cite{[N]}, we obtain the commutators in the group $[X,\mathrm{Sp(n)}]$.
\begin{thm}\label{d1}
\rm{ Let $X$ be a $CW$-complex with dim $X \leq 4n+2$ and let $\iota\colon \mathrm{N_n(X)}\rightarrow
\mathrm{Sp_n(X)}$ be as Theorem $1.1$ in \cite{[N]}. For $\alpha,\beta\in \mathrm{Sp_n(X)}$, the commutator
$[\alpha,\beta]$ in $\mathrm{Sp_n(X)}$ is given as
$$[\alpha, \beta]=\iota([\sum_{i+j=n+1}\alpha^*(y_{4i-1})\beta^*(y_{4j-1})]).\quad \Box$$
}
\end{thm}
Put $X'=S^{4m-1}\times Q_{n-m+1}$. Consider the following compositions
$$S^{4m-1}\times Q_{n-m+1} \overset{p_1}\longrightarrow S^{4m-1} \overset{\varepsilon_{m,n}}\longrightarrow
\mathrm{Sp(n)}, $$
$$S^{4m-1}\times Q_{n-m+1}\overset{p_2}\longrightarrow Q_{n-m+1} \overset{\epsilon_{m,n}}\longrightarrow
\mathrm{Sp(n)}, $$
where the maps $p_1$ and $p_2$ are the first and the second projections. Put $\varepsilon_{m,n} \circ p_1=\alpha$ and
$\epsilon_{m,n} \circ p_2=\beta$. The commutator $[\alpha, \beta]\in [X', \mathrm{Sp(n)}]$ is the following
composition
$$X' \overset{\bar{\Delta}}\longrightarrow X'\wedge X'  \overset{\alpha \wedge \beta}\longrightarrow
\mathrm{Sp(n)}\wedge
\mathrm{Sp(n)}\overset{\gamma}\longrightarrow \mathrm{Sp(n)},$$
where $\bar{\Delta}$ and $\gamma$ are the reduced diagonal map and the commutator map. Consider the following diagram
\begin{equation*}
\begin{tikzpicture}[baseline=(current bounding box.center)]
\matrix(m)[matrix of math nodes,
 row sep=3em, column sep=3em,
 text height=1.5ex, text depth=0.35ex]
 {X'& X'\times X'& X'& X& \mathrm{Sp(n)}\wedge \mathrm{Sp(n)}& \mathrm{Sp(n)}\\ & X'\wedge X'\\};
 \path[->,font=\scriptsize]
  (m-1-1) edge node[above] {$\Delta$}(m-1-2)
  (m-1-2) edge node[above] {$p_1\times p_2$} (m-1-3)
  (m-1-3) edge node[above] {$q$}(m-1-4)
      (m-1-4) edge node[above] {$\varepsilon_{m,n}\wedge\epsilon_{m,n}$} (m-1-5)
  (m-1-5) edge node[above] {$\gamma$} (m-1-6)
  (m-1-1) edge node[left] {$\bar{\Delta}$}(m-2-2)
  (m-2-2) edge node[right] {$\quad\alpha \wedge \beta$}(m-1-5);
 \end{tikzpicture}
\end{equation*}
where $\Delta$ is the diagonal map and $q\colon X'\rightarrow X$ is the quotient map. Since the composition
$$X'=S^{4m-1}\times Q_{n-m+1} \overset{\Delta}\longrightarrow X'\times X'  \overset{p_1\times p_2}\longrightarrow
S^{4m-1}\times Q_{n-m+1}=X' $$
is $1$, the Samelson product $\langle \varepsilon_{m,n}, \epsilon_{m,n}\rangle = \gamma \circ
(\varepsilon_{m,n}\wedge\epsilon_{m,n})$ is equal to the commutator $[\alpha,\beta]$ in the
group $[X, \mathrm{Sp(n)}]$. So by Theorem \ref{d1}, we have
$$[\alpha,\beta ]=\iota ([{\varepsilon_{m,n}}^*(y_{4m-1})\otimes {\epsilon_{m,n}}^*(y_{4n-4m+3})])=
\iota([a_m\sigma^{4m-1}\otimes \bar{y}_{4n-4m+3}]),$$
where by Bott integrality (see \cite{[BH]})
\begin{align*}
{\varepsilon_{m,n}}^*(y_{4m-1})=\pm a_m\sigma^{4m-1}, \qquad
a_m= \left\{ \begin{array}{ll}
2(2m-1)! & \textrm{if $m$ is even },\\
(2m-1)! & \textrm{if $m$ is odd}.\\
\end{array} \right.
\end{align*}
Therefore by Theorem \ref{c4}, the order of the Samelson product $\langle
\varepsilon_{m,n},\epsilon_{m,n} \rangle$ is equal to
\begin{align*}
\qquad \left\{ \begin{array}{ll}
\frac{(2n+1)!}{2(2n-2m+1)!(2m-1)!} & \textrm{if $m$ is even },\\\\
\frac{2(2n+1)!}{(2n-2m+1)!(2m-1)!} & \textrm{if $m$ is odd}. \quad \Box\\
\end{array} \right.
\end{align*}
\section{Proof of Theorem \ref{a2}}\label{e0}
In this section, we will prove Theorem \ref{a2}. Put $A=\Sigma^{4n-4m-5}Q_2$ and $X'=S^{4m}\wedge A$. First, we need
the following lemmas.
\begin{lem}\label{e1}
\rm{Let $\mu\colon\Sigma A\rightarrow \mathrm{Sp(n)}$ be an element of $\widetilde{KSp}(\Sigma^2 A)$ and $\mu'\colon
S^{4m}\wedge A\rightarrow \mathrm{Sp(n)}$ be the adjoint of the following composition
$$ S^{4m}\wedge \Sigma A \overset {\Sigma \varepsilon_{m,n}\wedge \mu}\longrightarrow \Sigma \mathrm{Sp(n)}\wedge
\mathrm{Sp(n)}\overset {[ev,ev]}\longrightarrow \mathrm{BSp(n)}.$$
Then there is a lift $\tilde{\mu}\colon S^{4m}\wedge A \rightarrow \Omega X_n$ of $\mu'$ with
${\tilde{\mu}}^*(b_{4n+2})=\pm a_m \sigma^{4m}\otimes \Sigma^{-1} {\mu}^*(y_{4(n-m+1)-1})$,} where $a_m$ is equal to
$2(2m-1)!$ if $m$ is even, and is equal to $(2m-1)!$ if $m$ is odd.
\end{lem}
\begin{proof}
By [Lemma $3.1$, 25], we know that there is a lift $\lambda\colon\Sigma \mathrm{Sp(n)}\wedge
\mathrm{Sp(n)}\longrightarrow X_{n}$ of $[ev,ev]$ such that $\lambda^*(\bar{y}_{4n+3})=\sum\limits_{i+j=n+1}\Sigma
y_{4i-1}\otimes y_{4j-1}$. Let $\tilde{\lambda}$ be the composition $\tilde{\lambda}\colon  S^{4m}\wedge \Sigma A
\overset {\Sigma \varepsilon_{m,n}\wedge \mu}\longrightarrow \Sigma \mathrm{Sp(n)}\wedge \mathrm{Sp(n)}\overset
{\lambda}\longrightarrow X_{n}$. We have
\begin{align*}
{\tilde{\lambda}}^*(\bar{y}_{4n+3})&=(\Sigma \varepsilon_{m,n}\wedge \mu)^*{\lambda}^*(\bar{y}_{4n+3})\\&= (\Sigma
\varepsilon_{m,n}\wedge \mu)^*(\sum\limits_{i+j=n+1}\Sigma y_{4i-1}\otimes y_{4j-1})\\&=
\pm a_m \sigma^{4m}\otimes {\mu}^* (y_{4(n-m+1)-1}),
\end{align*}
where ${\varepsilon_{m,n}}^*(y_{4m-1})=\pm a_m\sigma^{4m-1}$, $a_m$ is equal to $2(2m-1)!$ if $m$ is even and is equal
to $(2m-1)!$ if $m$ is odd. Now, we take $\tilde{\mu}\colon S^{4m}\wedge A \longrightarrow \Omega X_{n}$ to be the
adjoint of the following composition
$$\Sigma S^{4m}\wedge A \overset {S}\longrightarrow S^{4m}\wedge \Sigma A\overset {\tilde{\lambda}}\longrightarrow
X_{n},$$
that is $\tilde{\mu}\colon \mathrm{ad} (\tilde{\lambda}\circ S)$, where the map $S\colon\Sigma S^{4m}\wedge A
\longrightarrow S^{4m}\wedge \Sigma A $ is the swapping map and the map $\mathrm{ad}\colon [\Sigma S^{4m}\wedge A,
X_{n} ]\longrightarrow [S^{4m}\wedge A, \Omega X_{n}] $ is the adjunction. Then $\tilde{\mu}$ is a lift of $\mu$. Note
that
\begin{align*}
(\tilde{\lambda}\circ S)^*(\bar{y}_{4n+3})=S^* \circ \tilde{\lambda}^*(\bar{y}_{4n+3})&=
S^*(\pm a_m \sigma^{4m}\otimes {\mu}^*(y_{4(n-m+1)-1}))\\&= \pm a_m \sigma^{4m+1} \otimes \Sigma^{-1}
{\mu}^*(y_{4(n-m+1)-1}),
\end{align*}
thus we get ${\tilde{\mu}}^*(b_{4n+2})=\pm a_m \sigma^{4m}\otimes \Sigma ^{-1} {\mu}^* (y_{4(n-m+1)-1})$.
\end{proof}
Apply the functor $[\Sigma^{4n-4(m+1)}Q_2,\quad ]$ to the fibration sequence $(2.1)$, we get the following exact
sequence
\begin{align*}
[\Sigma^{4n-4(m+1)}Q_2,\mathrm{Sp(n)}]\overset{(\alpha_k)_*}\longrightarrow [\Sigma^{4n-5}Q_2,\mathrm{Sp(n)}]
&\longrightarrow[\Sigma^{4n-4(m+1)}Q_2,B\mathcal{G}_{k,m}(\mathrm{Sp(n)})]\\
&\longrightarrow [\Sigma^{4n-4(m+1)}Q_2, \mathrm{BSp(n)}]. \quad (**)
\end{align*}
In the following lemma, we show that $[\Sigma^{4n-4(m+1)}Q_2, \mathrm{BSp(n)}]$ is isomorphic to zero.
\begin{lem}\label{e2}
\rm{There is an isomorphism $[\Sigma^{4n-4(m+1)}Q_2, \mathrm{BSp(n)}]\cong 0$.}
\end{lem}
\begin{proof}
Since the dimension of $\Sigma^{4n-4(m+1)}Q_2$ is equal to $4(n-m)+3$, we have
$$[\Sigma^{4n-4(m+1)}Q_2, \mathrm{BSp(n)}]\cong [\Sigma^{4n-4(m+1)}Q_2, \mathrm{BSp(\infty)}] \cong
\widetilde{KSp}(\Sigma^{4n-4(m+1)}Q_2).$$
The cofibration sequence $ S^{4(n-m)-1}\rightarrow \Sigma^{4n-4(m+1)}Q_2\rightarrow S^{4(n-m)+3}$ induces the
following exact sequence
$$\cdots \rightarrow \widetilde{KSp}(S^{4(n-m)+3})\rightarrow \widetilde{KSp}(\Sigma^{4n-4(m+1)}Q_2)\rightarrow
\widetilde{KSp}(S^{4(n-m)-1})\rightarrow \cdots.$$
Since $\widetilde{KSp}(S^{4i-1})\cong 0$ for $i \geq 1$, we conclude $\widetilde{KSp}(\Sigma^{4n-4(m+1)}Q_2)\cong 0$.
\end{proof}
Therefore by Lemma \ref{e2}, we get the following lemma.
\begin{lem}\label{e3}
\rm{There is an isomorphism $[\Sigma^{4n-4(m+1)}Q_2, B\mathcal{G}_{k,m}(\mathrm{Sp(n)})]\cong
\mathrm{Coker}\hspace{1mm}(\alpha_k)_*$.} $\quad \Box$
\end{lem}
Note that by Theorem \ref{c4}, we know that $[\Sigma^{4n-5}Q_2,\mathrm{Sp(n)}]$ obtained by replacing $m = n-1$. In
what follows in this section, we will compute the Im $(\alpha_k)_*$ and then prove Theorem \ref{a2}. Consider the
map $\omega\colon \Sigma A\rightarrow \mathrm{Sp(n)}$, then by Lemma \ref{e1}, there is a lift $\tilde{\omega}$ of
$(\alpha_k)_*(\omega)$ such that we have
$${\tilde{\omega}}^*(b_{4n+2})=\pm k a_m\sigma^{4m}\otimes \Sigma ^{-1} {\omega}^* (y_{4(n-m+1)-1}).$$
We define the map $\beta_k\colon \widetilde{KSp}(\Sigma^2A) \rightarrow H^{4n+2}(X')$ by
$\beta_k(\omega)={\tilde{\omega}}^*(b_{4n+2})=b_{4n+2}\circ \tilde{\omega}$. Now, Consider the following commutative
diagram with exact columns and rows
\begin{equation*}
\begin{tikzpicture}[baseline=(current bounding box.center)]
\matrix(m)[matrix of math nodes,
 row sep=2em, column sep=2em,
 text height=1.5ex, text depth=0.15ex]
 {&{{\widetilde{KSp}}^{-2}(\Sigma^{4n-5} Q_2)}& &\\
 {[\Sigma^{4n-4(m+1)}Q_2,\mathrm{Sp(n)}]}&  H^{4n+2}(\Sigma^{4n-5} Q_2)& &\\
 {[\Sigma^{4n-4(m+1)}Q_2,\mathrm{Sp(n)}]}&{[\Sigma^{4n-5}Q_2,\mathrm{Sp(n)}]}& {[\Sigma^{4n-4(m+1)}Q_2,
 B\mathcal{G}_k]} &0\\
 & 0 && \\};
\path[->,font=\scriptsize]
  (m-3-1) edge node[above] {$(\alpha_k)_*$}(m-3-2)
  (m-2-1) edge node[above] {$\beta_k$}(m-2-2)
      (m-1-2) edge node[left] {$\vartheta'$} (m-2-2)
    (m-3-2) edge (m-3-3)
    (m-3-3) edge (m-3-4)
    (m-2-2) edge (m-3-2)
    (m-3-2) edge (m-4-2);
    \draw[double distance =2pt,font=\scriptsize]
    (m-2-1) -- (m-3-1);
 \end{tikzpicture}
\end{equation*}
where $[\Sigma^{4n-4(m+1)}Q_2,\mathrm{Sp(n)}]\cong\widetilde{KSp}(\Sigma^2A)$. We have
$\mathrm{Im}\hspace{1mm}(\alpha_k)_*=\mathrm{Im}\hspace{1mm}\beta_k/\mathrm{Im}\hspace{1mm}\vartheta'$. Note that by
Proposition \ref{c3}, we know that the $\mathrm{Im}\hspace{1mm}\vartheta'$ obtained by replacing $m = n-1$, thus we
need to calculate $\mathrm{Im}\hspace{1mm}\beta_k$.\par
Let $a\colon \Sigma Q_2\rightarrow \mathrm{BSp(\infty)}$ be the adjoint of the composition of the inclusions $Q_2
\rightarrow \mathrm{Sp(2)} \rightarrow \mathrm{Sp(\infty)}$ and also $b\colon \Sigma Q_2\rightarrow
\mathrm{BSp(\infty)}$ be the pinch map onto the top cell $q\colon \Sigma Q_2\rightarrow S^8$ followed by a generator
of $\pi_8(\mathrm{BSp(\infty)})\cong \mathbb{Z}$. Note that $\widetilde{KSp} (\Sigma Q_2)$ is a free abelian group
with a basis $a , b$. Then we have
$$\mathrm{ch}(c'(a))=\Sigma y_3-\frac{1}{6} \Sigma y_7, \qquad \mathrm{ch}(c'(b))=-2\Sigma y_7.$$
Let $\theta_1= \textbf{q}(\zeta_{2(n-m+1)-4}c'(a))\in \widetilde{KSp}(\Sigma^{4(n-m+1)-7}Q_2)$, where
$\textbf{q}\colon K \rightarrow KSp$ is the quaternionization. Also, let $\theta_2\colon\Sigma^{4(n-m+1)-7}Q_2
\rightarrow \mathrm{BSp(\infty)}$ be the composite of the pinch map to the top cell $\Sigma^{4(n-m+1)-7}Q_2
\rightarrow S^{4(n-m+1)}$ and a generator of $\pi_{4(n-m+1)}(BSp(\infty))\cong \mathbb{Z}$. Then $\widetilde {KSp}
(\Sigma^2A)$ is a free abelian group with a basis $\theta_1, \theta_2$. We have
\begin{align*}
\mathrm{ch}(c'(\theta_1))= \left\{ \begin{array}{ll}
\Sigma^{4(n-m+1)-7} y_3-\frac{1}{6} \Sigma^{4(n-m+1)-7} y_7  & \textrm{if $n-m+1$ is even},\\\\
2\Sigma^{4(n-m+1)-7} y_3+\frac{1}{3} \Sigma^{4(n-m+1)-7} y_7& \textrm{if $n-m+1$ is odd},\\
\end{array} \right.
\end{align*}
and
\begin{align*}
 \mathrm{ch}(c'(\theta_2))= \left\{ \begin{array}{ll}
-2\Sigma^{4(n-m+1)-7} y_7 & \textrm{if $n-m+1$ is even},\\\\
\Sigma^{4(n-m+1)-7} y_7 & \textrm{if $n-m+1$ is odd}.\\
\end{array} \right.
\end{align*}
\begin{lem}\label{e4}
\rm{ In case $m$ is even,
\begin{displaymath}
\mathrm{Im}\hspace{1mm}\beta_k \cong \left\{ \begin{array}{ll}
\mathbb{Z}\{\frac{2}{3}(2m-1)!(2(n-m+1)-1)! k\sigma^{4n-5}\otimes y_7\} & \textrm{if $n$ is even},\\\\
\mathbb{Z}\{\frac{1}{3} (2m-1)!(2(n-m+1)-1)! k\sigma^{4n-5}\otimes y_7\} & \textrm{if $n$ is odd },\\
\end{array} \right.
\end{displaymath}
and in case $m$ is odd,
\begin{displaymath}
\mathrm{Im}\hspace{1mm}\beta_k \cong \left\{ \begin{array}{ll}
\mathbb{Z}\{\frac{1}{6}(2m-1)!(2(n-m+1)-1)! k\sigma^{4n-5}\otimes y_7\} & \textrm{if $n$ is even},\\\\
\mathbb{Z}\{\frac{1}{3}(2m-1)!(2(n-m+1)-1)! k\sigma^{4n-5}\otimes y_7\} & \textrm{if $n$ is odd }.\\
\end{array} \right.
\end{displaymath}
}
\end{lem}
\begin{proof}
For $i=1,2$, by Lemma \ref{e1}, $(\beta_k)_*(\theta_i)$ has a lift ${\tilde{\theta}}_{i,k}\colon S^{4m}\wedge
A\rightarrow \Omega X_n$ such that
$${{\tilde{\theta}}_{i,k}}^*(b_{4n+2})=k a_m\sigma^{4m}\otimes \Sigma ^{-1} {\theta_i}^* (y_{4(n-m+1)-1}).$$
Since $y_{4i-1}=\sigma (q_i)$ and $q_i={c'}^*(c_{2i})$, so we get
$${{\tilde{\theta}}_{i,k}}^*(b_{4n+2})=k a_m \sigma^{4m}\otimes \Sigma ^{-2} q_{(n-m+1)}(\theta_i)=k a_m
\sigma^{4m}\otimes \Sigma ^{-2} c_{2(n-m+1)}(c'(\theta_i)). $$
On the other hand we have
\begin{align*}
c_{2(n-m+1)}(c'(\theta_1))= \left\{ \begin{array}{ll}
\frac{1}{6}(2(n-m+1)-1)! \Sigma^{4(n-m+1)-7} y_7  & \textrm{if $n-m+1$ is even},\\\\
\frac{1}{3}(2(n-m+1)-1)! \Sigma^{4(n-m+1)-7} y_7& \textrm{if $n-m+1$ is odd},\\
\end{array} \right.
\end{align*}
\begin{align*}
c_{2(n-m+1)}(c'(\theta_2))= \left\{ \begin{array}{ll}
2(2(n-m+1)-1)!\Sigma^{4(n-m+1)-7} y_7 & \textrm{if $n-m+1$ is even},\\\\
(2(n-m+1)-1)!\Sigma^{4(n-m+1)-7} y_7 & \textrm{if $n-m+1$ is odd}.\\
\end{array} \right.
\end{align*}
Therefore in case $m$ is even, we get
\begin{align*}
{{\tilde{\theta}}_{1,k}}^*(b_{4n+2})= \left\{ \begin{array}{ll}
k 2(2m-1)!\sigma^{4m}\otimes\frac{1}{3}(2(n-m+1)-1)! \Sigma^{4(n-m+1)-9} y_7 & \textrm{if $n$ is even},\\\\
k 2(2m-1)!\sigma^{4m}\otimes \frac{1}{6}(2(n-m+1)-1)! \Sigma^{4(n-m+1)-9} y_7  & \textrm{if $n$ is odd},\\
\end{array} \right.
\end{align*}
\begin{align*}
{{\tilde{\theta}}_{2,k}}^*(b_{4n+2})= \left\{ \begin{array}{ll}
k2(2m-1)!\sigma^{4m}\otimes(2(n-m+1)-1)! \Sigma^{4(n-m+1)-9} y_7 & \textrm{if $n$ is even},\\\\
k2(2m-1)!\sigma^{4m}\otimes 2(2(n-m+1)-1)! \Sigma^{4(n-m+1)-9} y_7 & \textrm{if $n$ is odd},\\
\end{array} \right.
\end{align*}
and in case $m$ is odd, we get
\begin{align*}
{{\tilde{\theta}}_{1,k}}^*(b_{4n+2})= \left\{ \begin{array}{ll}
k (2m-1)!\sigma^{4m}\otimes \frac{1}{6}(2(n-m+1)-1)! \Sigma^{4(n-m+1)-9} y_7 & \textrm{if $n$ is even},\\\\
 k (2m-1)!\sigma^{4m}\otimes\frac{1}{3}(2(n-m+1)-1)! \Sigma^{4(n-m+1)-9} y_7 & \textrm{if $n$ is odd},\\
\end{array} \right.
\end{align*}
\begin{align*}
{{\tilde{\theta}}_{2,k}}^*(b_{4n+2})= \left\{ \begin{array}{ll}
k(2m-1)!\sigma^{4m}\otimes 2(2(n-m+1)-1)! \Sigma^{4(n-m+1)-9} y_7 & \textrm{if $n$ is even},\\\\
k(2m-1)!\sigma^{4m}\otimes(2(n-m+1)-1)! \Sigma^{4(n-m+1)-9} y_7  & \textrm{if $n$ is odd}.\\
\end{array} \right.
\end{align*}
So, in case $m$ is even, $\mathrm{Im}\hspace{1mm}\beta_k$ is isomorphic to $\mathbb{Z}\{\alpha_1, \beta_1\}$ if $n$ is
even, and is isomorphic to $\mathbb{Z}\{\alpha'_1, \beta'_1\}$ if $n$ is odd, where
\begin{align*}
&\alpha_1=\frac{2}{3}(2m-1)!(2(n-m+1)-1)! k\sigma^{4n-5}\otimes y_7,\\
&\beta_1 = 2(2m-1)!(2(n-m+1)-1)! k\sigma^{4n-5}\otimes y_7, \\
&\alpha'_1=\frac{1}{3} (2m-1)!(2(n-m+1)-1)! k\sigma^{4n-5}\otimes y_7,\\
&\beta'_1 = 4(2m-1)! (2(n-m+1)-1)! k\sigma^{4n-5}\otimes y_7,
\end{align*}
and in case $m$ is odd, $\mathrm{Im}\hspace{1mm}\beta_k$ is isomorphic to $\mathbb{Z}\{\alpha_2, \beta_2\}$ if $n$ is
even, and is isomorphic to $\mathbb{Z}\{\alpha'_2, \beta'_2\}$ if $n$ is odd, where
\begin{align*}
&\alpha_2=\frac{1}{6}(2m-1)!(2(n-m+1)-1)! k\sigma^{4n-5}\otimes y_7,\\
&\beta_2 = 2(2m-1)!(2(n-m+1)-1)! k\sigma^{4n-5}\otimes y_7, \\
&\alpha'_2=\frac{1}{3}(2m-1)!(2(n-m+1)-1)! k\sigma^{4n-5}\otimes y_7,\\
&\beta'_2 = (2m-1)!(2(n-m+1)-1)! k\sigma^{4n-5}\otimes y_7.
\end{align*}
Also, in case $m$ is even, $\mathrm{Im}\hspace{1mm}\beta_k$ is generated by $\beta_1-2\alpha_1
=\frac{2}{3}(2m-1)!(2(n-m+1)-1)! k\sigma^{4n-5}\otimes y_7$ if $n$ is even, and is generated by $\beta'_1-11\alpha'_1=
\frac{1}{3} (2m-1)!(2(n-m+1)-1)! k\sigma^{4n-5}\otimes y_7$ if $n$ is odd. In case $m$ is odd,
$\mathrm{Im}\hspace{1mm}\beta_k$ is generated by $\beta_2-11\alpha_2=\frac{1}{6}(2m-1)!(2(n-m+1)-1)!
k\sigma^{4n-5}\otimes y_7$ if $n$ is even, and is generated by $\beta'_2-2\alpha'_2= \frac{1}{3}(2m-1)!(2(n-m+1)-1)!
k\sigma^{4n-5}\otimes y_7$ if $n$ is odd. Therefore in case $m$ is even, we can conclude that
\begin{displaymath}
\mathrm{Im}\hspace{1mm}\beta_k \cong \left\{ \begin{array}{ll}
\mathbb{Z}\{\frac{2}{3}(2m-1)!(2(n-m+1)-1)! k\sigma^{4n-5}\otimes y_7\} & \textrm{if $n$ is even},\\\\
\mathbb{Z}\{\frac{1}{3} (2m-1)!(2(n-m+1)-1)! k\sigma^{4n-5}\otimes y_7\} & \textrm{if $n$ is odd },\\
\end{array} \right.
\end{displaymath}
and in case $m$ is odd, we can conclude that
\begin{displaymath}
\mathrm{Im}\hspace{1mm}\beta_k \cong \left\{ \begin{array}{ll}
\mathbb{Z}\{\frac{1}{6}(2m-1)!(2(n-m+1)-1)! k\sigma^{4n-5}\otimes y_7\} & \textrm{if $n$ is even},\\\\
\mathbb{Z}\{\frac{1}{3}(2m-1)!(2(n-m+1)-1)! k\sigma^{4n-5}\otimes y_7\} & \textrm{if $n$ is odd }.\\
\end{array} \right.
\end{displaymath}
\end{proof}
Therefore, by Proposition \ref{c3}, and Lemma \ref{e4}, we obtain the following proposition.
\begin{prop}\label{e5}
\rm{The following hold:\\
\textbf {Case one:} if $m$ is an even integer, then
\begin{displaymath}
\mathrm{Im}\hspace{1mm} (\alpha_k)_*\cong \left\{ \begin{array}{ll}
\mathbb{Z}/((2n+1)!/(3(k,t))) & \textrm{if $n$ is even },\\\\
\mathbb{Z}/((2n+1)!/(6(k,t))) & \textrm{if $n$ is odd },\\
\end{array} \right.
\end{displaymath}
\textbf {Case two:} if $m$ is an odd integer, then
\begin{displaymath}
\mathrm{Im}\hspace{1mm} (\alpha_k)_*\cong \left\{ \begin{array}{ll}
\mathbb{Z}/((2n+1)!/(3(k,4t))) & \textrm{if $n$ is even },\\\\
\mathbb{Z}/((2n+1)!/(6(k,t))) & \textrm{if $n$ is odd }. \quad \Box \\
\end{array} \right.
\end{displaymath}
}
\end{prop}
Therefore by Theorem \ref{c4}, Lemma \ref{e3} and Proposition \ref{e5}, we get the following theorem.
\begin{thm}\label{e6}
\rm{There are following isomorphisms\\
\textbf {Case one:} if $m$ is an even integer, then $[\Sigma^{4n-4(m+1)}Q_2, B\mathcal{G}_{k,m}(\mathrm{Sp(n)})]\cong
\mathbb{Z}_{(k,t)},$\\
\textbf {Case two:} if $m$ is an odd integer, then
\begin{displaymath}
[\Sigma^{4n-4(m+1)}Q_2, B\mathcal{G}_{k,m}(\mathrm{Sp(n)})]\cong \left\{ \begin{array}{ll}
\mathbb{Z}_{(k,4t)} & \textrm{if $n$ is even },\\\\
\mathbb{Z}_{(k,t)} & \textrm{if $n$ is odd },
\end{array} \right. \quad \Box
\end{displaymath}
}
\end{thm}
$\textbf{Proof of Theorem \ref{a2}}$\\
We write $|H|$ for the order of a group $H$. Consider the exact sequence $(**)$. Suppose that
$\mathcal{G}_{k,m}(\mathrm{Sp(n)})\simeq \mathcal{G}_{k',m}(\mathrm{Sp(n)})$, then there is an isomorphism of groups
$$[\Sigma^{4n-4(m+1)}Q_2, B\mathcal{G}_{k,m}(\mathrm{Sp(n)})]\cong [\Sigma^{4n-4(m+1)}Q_2,
B\mathcal{G}_{k',m}(\mathrm{Sp(n)})].$$
Thus we have
$$|[\Sigma^{4n-4(m+1)}Q_2, B\mathcal{G}_{k,m}(\mathrm{Sp(n)})]|=|[\Sigma^{4n-4(m+1)}Q_2,
B\mathcal{G}_{k',m}(\mathrm{Sp(n)})]|.$$
Therefore, by Theorem \ref{e6}, we can conclude Theorem \ref{a2}.$\quad \Box$\\

\textbf{Competing interests:} The author declares that he has no competing interests.


School of Mathematics, Institute for Research in Fundamental Sciences (IPM), P.O. Box
19395-5746, Tehran, Iran\\
E-mail: sj.mohammadi@ipm.ir
\end{document}